\documentclass[12pt, reqno]{amsart}
\usepackage{amsmath, amsthm, amscd, amsfonts, amssymb, graphicx, color}
\usepackage{eurosym}

\newtheorem{theorem}{Theorem}[section]
\newtheorem{lemma}[theorem]{Lemma}

\theoremstyle{definition}

\newtheorem{conjecture}[theorem]{Conjecture}

\theoremstyle{remark}
\newtheorem{remark}[theorem]{Remark}
\numberwithin{equation}{section}

\begin{document}
\title[Sharp coefficients bounds for Starlike functions]{Sharp coefficients
bounds for Starlike functions associated with Gregory coefficients}
\author{\textbf{\ Erhan Deniz$^1,$} \textbf{\ Sercan Kaz\i mo\u glu$^1 $} and \textbf{\ H. M. Srivastava$^{2,3,4,5,6}$}
}
\address{$^1 $Department of Mathematics, Faculty of Science and Letters, Kafkas University, Kars-TURKEY  \newline 
$^2$Department of Mathematics and Statistics,
University of Victoria, Victoria, British Columbia V8W 3R4, Canada   \newline
$^3$Department of Medical Research, China Medical University Hospital, China Medical University, Taichung 40402,
Taiwan  \newline 
$^4$Center for Converging Humanities, Kyung Hee University, 26 Kyungheedae-ro, Dongdaemun-gu, Seoul 02447, Republic of Korea  \newline
$^5$Department of Mathematics and Informatics, Azerbaijan University, 71 Jeyhun Hajibeyli Street, AZ1007 Baku, Azerbaijan  \newline 
$^6$Section of Mathematics, International Telematic University Uninettuno, I-00186 Rome, Italy}
\email{edeniz36@gmail.com, srcnkzmglu@gmail.com and harimsri@math.uvic.ca}

\begin{abstract}
In this paper we introduced the class $\mathcal{S}_{G}^{\ast }$ of analytic
functions which is related with starlike functions and generating function
of Gregory coefficients. By using bounds on some coefficient functionals for
the family of functions with positive real part, we obtain for functions in
the class $\mathcal{S}_{G}^{\ast }$ several sharp coefficient bounds on the
first six coeffcients and also further sharp bounds on the corresponding
Hankel determinants.
\end{abstract}

\maketitle

\section{Introduction}

Let $\mathcal{A}$ be the class of functions $f$ which are analytic in the
open unit disc $\mathcal{U}=\left\{ z:\left\vert z\right\vert <1\right\} $
and normalized by the conditions $f\left( 0\right) =f^{\prime }\left(
0\right) -1=0.$ Let us denote by $\mathcal{S}$ the subclass of $\mathcal{A}$
containing functions which are univalent in $\mathcal{U}$. An analytic
function $f$ is subordinate to an analytic function $g$ (written as $f\prec
g)$ if there exists an analytic function $w$ with $w\left( 0\right) =0$ and$%
\ \left\vert w\left( z\right) \right\vert <1$ for $z\in \mathcal{U}$ such
that $f\left( z\right) =g\left( w\left( z\right) \right) .$ In particular,
if $g$ is univalent in $\mathcal{U},\ $then $f\left( 0\right) =g\left(
0\right) $ and $f\left( \mathcal{U}\right) \subset g\left( \mathcal{U}%
\right) .$ 

In 1992, Ma and Minda \cite{Ma} gave a unified presentation of various
subclasses of starlike and convex functions by replacing the subordinate
function $(1+z)\diagup (1-z)$ by a more general analytic function $\varphi $
with positive real part and normalized by the conditions $\varphi (0)=1,$ $%
\varphi ^{\prime }(0)>0$ and $\varphi $ maps $\mathcal{U}$ onto univalently
a region starlike with respect to $1$ and symmetric with respect to the real
axis. They introduced the following general class that envelopes several
well-known classes as special cases: $\mathcal{S}^{\ast }[\varphi ]=\left\{
f\in \mathcal{A}:\text{ }zf^{\prime }\left( z\right) \diagup f\left(
z\right) \prec \varphi (z)\right\} .$ In literature, the functions belonging
to this class is called Ma-Minda starlike function. For $-1\leq B<A\leq 1$, $%
\mathcal{S}^{\ast }[\left( 1+Az\right) \diagup \left( 1+Bz\right) ]:=%
\mathcal{S}^{\ast }[A,B]$ are called the Janowski starlike functions,
introduced by Janowski \cite{Ja}. The class $\mathcal{S}^{\ast }(\beta )$ of
starlike functions of order $\beta $ $\left( 0\leq \beta <1\right) $ defined
by taking $\varphi (z)=(1+(1-2\beta )z)\diagup (1-z).$ Note that $\mathcal{S}%
^{\ast }=\mathcal{S}^{\ast }(0)$ is the classical class of starlike
functions. By taking $\varphi (z)=1+2\diagup \pi ^{2}(\log \left( (1+\sqrt{z}%
)\diagup (1-\sqrt{z})\right) ^{2},$ we obtain the class $\mathcal{S}_{p}$ of
parabolic starlike functions, introduced by R\o nning \cite{Ronning}. $\mathcal{S}%
^{\ast }[\beta ,-\beta ]:=\widetilde{\mathcal{S}^{\ast }}(\beta )=\{f\in 
\mathcal{A}:|zf^{\prime }\left( z\right) \diagup f\left( z\right) -1|<\beta
|zf^{\prime }\left( z\right) \diagup f\left( z\right) +1|\}$ has been
studied in \cite{Ali1,Ali2}. 

Recently, the coefficient problem for $\mathcal{S}^{\ast }[\varphi ]$ in the
case of generated functions of some well-known numbers of $\varphi $ studied
by \cite{Cho,Dz,Dz2,Kum,Rav}. For example, the case when $\varphi $ is
defined by 
\begin{equation*}
\varphi (z)=\frac{1+\tau ^{2}z^{2}}{1-\tau z-\tau ^{2}z^{2}}\text{ \ \ }%
(z\in \mathcal{%
\mathbb{C}
},\;\tau =(1-\sqrt{5})/2),
\end{equation*}%
which generated function of Fibonacci numbers, coefficient problem studied
by Dziok and co-authors \cite{Dz,Dz2}. Similarly, the function 
\begin{equation*}
\varphi (z)=e^{e^{z}-1}\text{ \ \ }(z\in \mathcal{%
\mathbb{C}
})
\end{equation*}%
is generated function of Bell numbers. The coefficient problem for this
function studied by Kumar and co-authors \cite{Kum}. Recently, Mendiratta et al. \cite{Mendiratta} and, Goel and Kumar \cite{Goel} obtained the structural formula,
inclusion relations, coefficient estimates, growth and distortion results,
subordination theorems and various radii constants for the exponential function%
\begin{equation*}
\varphi (z)=e^{z}\text{ \ \ }(z\in \mathcal{%
\mathbb{C}
})
\end{equation*}%
and the Sigmoid function%
\begin{equation*}
\varphi (z)=\frac{2}{1+e^{-z}}\text{ \ \ }(z\in \mathcal{%
\mathbb{C}
})
\end{equation*}%
respectively. In 2021, Deniz \cite{Deniz} has studied sharp coefficient
problem for the function%
\begin{equation*}
\varphi (z)=e^{z+\frac{\lambda }{2}z^{2}}\text{ \ \ }(z\in \mathcal{%
\mathbb{C}
}\text{, }\lambda \geq 1),
\end{equation*}%
which is generated function of generalized telephone numbers.

Motivated by the above-cited works, we consider the function $\varphi $ for
which $\varphi (\mathcal{U})$ is starlike with respect to $1$ and whose
coefficients is the Gregory coefficients. Gregory coefficients are
decreasing rational numbers $1/2,1/12,$ $1/24,19/720,...,$ which are similar in
their role to the Bernoulli numbers and appear in a large number of
problems, especially in those related to the numerical analysis and to the
number theory. They first appeared in the works of the Scottish
mathematician James Gregory in 1671 and were subsequently rediscovered many
times. Among the famous mathematicians who rediscovered them, we find
Laplace, Mascheroni, Fontana, Bessel, Clausen, Hermite, Pearson and Fisher.
Furthermore, Gregory's coefficients can rightly be considered one of the
most frequently rediscovered in mathematics (the last rediscovery dates back
to our century), the reason for which in literature we can find them under
various names (e.g. reciprocal logarithmic numbers, Bernoulli numbers of the
second kind, Cauchy numbers, etc.). For the same
reason, their authorship is often attributed to various mathematicians. In
this paper, authors considered the generating function of the Gregory
coefficients $G_{n}$ (see \cite{Berezin,Phillips}) as follows:%
\begin{equation*}
\frac{z}{\ln \left( 1+z\right) }=\sum\limits_{n=0}^{\infty
}G_{n}z^{n}\;\;\;(\left\vert z\right\vert <1).
\end{equation*}

Clearly, $G_{n}$ for some values of $n$ as $G_{0}=1,$ $G_{1}=\frac{1}{2},$ $%
G_{2}=-\frac{1}{12},$ $G_{3}=\frac{1}{24},$ $G_{4}=-\frac{19}{720},$ $G_{5}=%
\frac{3}{160},$ and $G_{6}=-\frac{863}{60480}.$ We now consider the function 
$\Psi (z):=\frac{z}{\ln \left( 1+z\right) }$ with its domain of definition
as the open unit disk $\mathcal{U}$. In connection with this function, we
define the class $$\mathcal{S}_{G}^{\ast }:=\left\{ f:\;f\in \mathcal{S}\text{
and }zf^{\prime }(z)/f(z)\prec \Psi (z)\right\} .$$ 
Graphic of $\Psi (%
\mathcal{U})$ as follows: 
\begin{figure}[h]
\begin{tabular}{c}
\includegraphics[scale=0.8]{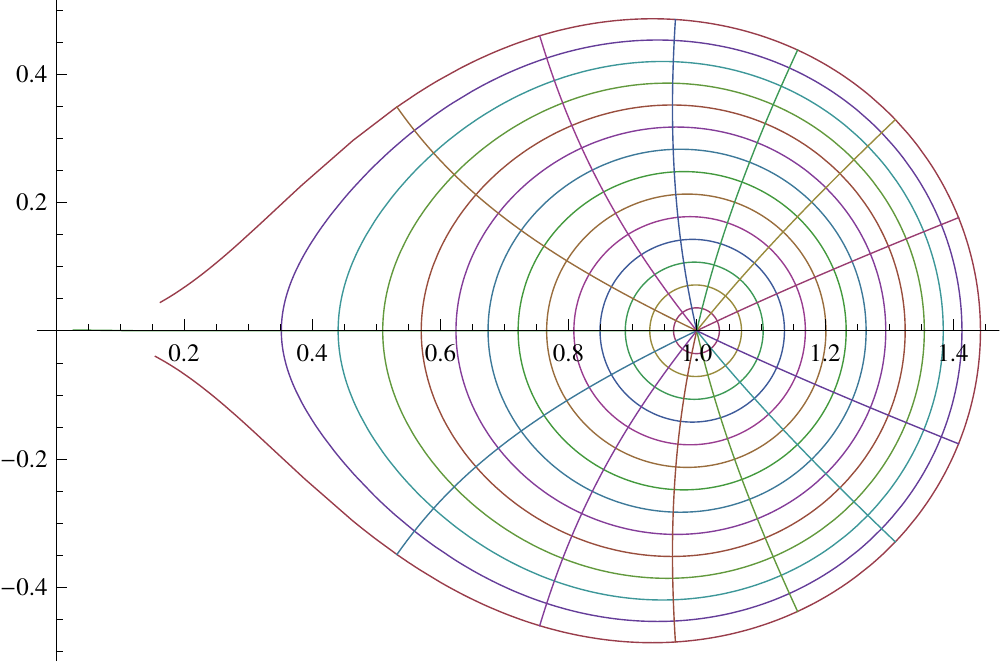}%
\end{tabular}%
\end{figure}

\section{Preliminary Results}

The following lemmas is needed for the main results.

\begin{lemma}
\label{lmma1}\cite{Lib} If $p(z)=1+p_{1}z+p_{2}z^{2}+p_{3}z^{3}+\cdots \in 
\mathcal{P}$ $(p_{1}\geq 0),$ then%
\begin{eqnarray}
2p_{2} &=&p_{1}^{2}+x(4-p_{1}^{2})  \label{eq21} \\
4p_{3}
&=&p_{1}^{3}+2(4-p_{1}^{2})p_{1}x-p_{1}(4-p_{1}^{2})x^{2}+2(4-p_{1}^{2})%
\left( 1-\left\vert x\right\vert ^{2}\right) y,  \label{eq22}
\end{eqnarray}%
for some $x,y\in 
\mathbb{C}
$ with $\left\vert x\right\vert \leq 1 \text{ and }\left\vert y\right\vert
\leq 1.$
\end{lemma}

\begin{lemma}
\label{lmma2} If $p(z)=1+p_{1}z+p_{2}z^{2}+p_{3}z^{3}+\cdots \in \mathcal{P}$
$(p_{1}\geq 0),$ then 
\begin{equation}  \label{eq24}
\left|p_n\right|\leq 2~~ \left(n\geq 1\right) ,
\end{equation}
and if $Q\in \left[0,1\right]$ and $Q\left(2Q-1\right)\leq R\leq Q,$ then 
\begin{equation}  \label{eq25}
\left|p_3-2Qp_1p_2+Rp_1^3\right|\leq 2.
\end{equation}
Also 
\begin{eqnarray}
\left|p_{n+k}-\mu p_np_k\right|&\leq&
2\max\left\{1,\left|2\mu-1\right|\right\}  \notag \\
&=&2\left\{ 
\begin{array}{ll}
1, & \text{ for } 0\leq \mu \leq 1, \\ 
\left|2\mu-1\right|, & \text{ otherwise }%
\end{array}
\right. .  \label{eq26}
\end{eqnarray}
The inequalities (\ref{eq24}), (\ref{eq25}) and (\ref{eq26}) are taken from 
\cite{Caratheodory,Lib} and \cite{Pom44}, respectively.
\end{lemma}

\begin{lemma}[see \protect\cite{Rav}]
\label{lmma3} Let $\tau,~\psi,~\rho$ and $\varsigma$ satify the inequalities 
$0<\tau <1,~0<\varsigma<1$ and 
\begin{equation}  \label{eq27}
\begin{split}
 8\varsigma \left(1-\varsigma\right)\left[\left(\tau\psi-2\rho%
\right)^2+\left(\tau\left(\varsigma+\tau\right)-\psi \right)^2 \right]+\tau
\left(1-\tau\right)\left(\psi-2\varsigma\tau\right)^2 \\  \leq
4\varsigma\tau^2\left(1-\tau\right)^2\left(1-\varsigma\right).
\end{split}
\end{equation}
If $p(z)=1+p_{1}z+p_{2}z^{2}+p_{3}z^{3}+\cdots \in \mathcal{P}$ 
$(p_{1}\geq 0),$ then 
\begin{equation}  \label{eq28}
\left|\rho p_1^4+\varsigma p_2^2+2\tau p_1p_3-\frac{3}{2}\psi
p_1^2p_2-p_4\right|\leq 2.
\end{equation}
\end{lemma}

\begin{lemma}
\label{l3}\cite{Oh} Let $\overline{\mathcal{U}}=\left\{ z:\left\vert
z\right\vert \leq 1\right\} .$ Also, for any real numbers $a,$ $b$ and $c,$
let the quantity $Y(a,b,c)=\max_{z\in \overline{\mathcal{U}}}\left\{
\left\vert a+bz+cz^{2}\right\vert +1-\left\vert z\right\vert ^{2}\right\} .$
If $ac\geq 0,$ then%
\begin{equation*}
Y(a,b,c)=\left\{ 
\begin{array}{cc}
\left\vert a\right\vert +\left\vert b\right\vert +\left\vert c\right\vert & 
\left\vert b\right\vert \geq 2(1-\left\vert c\right\vert ) \\ 
1+\left\vert a\right\vert +\frac{b^{2}}{4(1-\left\vert c\right\vert )} & 
\left\vert b\right\vert <2(1-\left\vert c\right\vert )%
\end{array}%
\right. .
\end{equation*}%
Furthermore, if $ac<0,$ then%
\begin{equation*}
Y(a,b,c)=\left\{ 
\begin{array}{cc}
1-\left\vert a\right\vert +\frac{b^{2}}{4(1-\left\vert c\right\vert )} & 
\left( -4ac(c^{-2}-1)\leq b^{2};\;\left\vert b\right\vert <2(1-\left\vert
c\right\vert )\right) \\ 
1+\left\vert a\right\vert +\frac{b^{2}}{4(1+\left\vert c\right\vert )} & 
b^{2}<\min \left\{ 4(1+\left\vert c\right\vert )^{2},-4ac(c^{-2}-1)\right\}
\\ 
R(a,b,c) & \left( \text{otherwise}\right)%
\end{array}%
\right. ,
\end{equation*}%
where%
\begin{equation*}
R(a,b,c)=\left\{ 
\begin{array}{cc}
\left\vert a\right\vert +\left\vert b\right\vert -\left\vert c\right\vert & 
\left( \left\vert c\right\vert \left( \left\vert b\right\vert +4\left\vert
a\right\vert \right) \leq \left\vert ab\right\vert )\right) \\ 
-\left\vert a\right\vert +\left\vert b\right\vert +\left\vert c\right\vert & 
\left( \left\vert ab\right\vert \leq \left\vert c\right\vert \left(
\left\vert b\right\vert -4\left\vert a\right\vert \right) \right) \\ 
\left( \left\vert a\right\vert +\left\vert c\right\vert \right) \sqrt{1-%
\frac{b^{2}}{4ac}} & \left( \text{otherwise}\right)%
\end{array}%
\right. .
\end{equation*}
\end{lemma}

\section{Coefficient Estimates}

Finding the upper bound for coefficients have been one of the central topic
of research in geometric function theory as it gives several properties of
functions. Therefore, we will be interested in the following problem in this
section.

Problem: Find $\sup \left\vert a_{n}\right\vert $ $(n=2,3,\cdots )$ for
certain sublasses of univalent functions. In particular, bound for the
second coefficient gives growth and distortion theorems for functions in the
class. Another one is the coefficient problem related with Hankel
determinants. Similarly, using the Hankel determinants (which also deals
with the bound on coefficients), Cantor \cite{Ca} proved that
\textquotedblleft if ratio of two bounded analytic functions in $\mathcal{U}$%
, then the function is rational\textquotedblright . The Hankel determinants 
\cite{Noo}$\;H_{q}(n)\;(n=1,2,...,\;q=1,2,...)$ of the function $f$\ are
defined by%
\begin{equation*}
H_{q}(n)=\left\vert 
\begin{array}{cccc}
a_{n} & a_{n+1} & ... & a_{n+q-1} \\ 
a_{n+1} & a_{n+2} & ... & a_{n+q} \\ 
\vdots & \vdots & \ddots & \vdots \\ 
a_{n+q-1} & a_{n+q} & ... & a_{n+2q-2}%
\end{array}%
\right\vert \text{ \ \ }(a_{1}=1).
\end{equation*}

This determinant was discussed by several authors with $q=2$. For example,
we know that the functional $H_{2}(1):=a_{3}-a_{2}^{2}\;$is known as the
Fekete-Szeg\"{o} functional and they consider the further generalized
functional $a_{3}-\mu a_{2}^{2},\;$where $\mu $ is some real number.
Estimating for the upper bound of $\left\vert a_{3}-\mu a_{2}^{2}\right\vert 
$ is known as the Fekete-Szeg\"{o} problem. In 1969, Keogh and Merkes \cite%
{Keo} solved the Fekete-Szeg\"{o} problem for the class $\mathcal{S}^{\ast }$%
. The second Hankel determinant $H_{2}(2)$ is given by $H_{2}(2):=\left\vert
a_{2}a_{4}-a_{3}^{2}\right\vert .$ The bounds for the second Hankel
determinant $H_{2}(2)$ obtained for the class $\mathcal{S}^{\ast }$ in \cite%
{Jan}. Lee et al. \cite{Lee} established the sharp bound to $%
\left\vert H_{2}(2)\right\vert $ by generalizing their classes using
subordination. Moreover, the quantity given by $%
H_{3}(1):=a_{3}(a_{2}a_{4}-a_{3}^{2})-a_{4}(a_{4}-a_{2}a_{3})+a_{5}(a_{3}-a_{2}^{2}) 
$ is called third Hankel determinant. Zaprawa \cite{Zap} proved that $%
\left\vert H_{3}(1)\right\vert <1$ for $f\in \mathcal{S}^{\ast }.$ In 2019,
Kwon, Lecko and Sim \cite{Kls} improved the result of Zaprawa as $\left\vert
H_{3}(1)\right\vert <8\diagup 9.$ This result the best known upper bound of $%
\left\vert H_{3}(1)\right\vert $ for the class $\mathcal{S}^{\ast }.$ Recently, Hankel determinants and Fekete-Szeg\" o problem have been considered in many papers of Srivastava and his co-authors (see, for example, \cite{Shi,Srivastava21,Srivastava22116,Srivastava2223,Srivastava23}) and Deniz and his co-authors (see, \cite{Denizbudak,kazimoglu}).

In this paper, we seek sharp upper bounds for the coefficients $a_{j}$ $%
(j=2,3,4,5,6)$ and functionals $H_{2}(2)$ and $H_{3}(1)\;$for functions $f\;$%
belonging to the class $\mathcal{S}_{G}^{\ast }.$

\begin{theorem}
\label{t1} Let $f(z)=z+a_{2}z^{2}+a_{3}z^{3}+\cdots \in \mathcal{S}%
_{G}^{\ast }.$ Then,%
\begin{eqnarray*}
\left\vert a_{n}\right\vert &\leq &\frac{1}{2(n-1)}\text{ \ \ }(n=2,3,4,5) \\
\left\vert {a_{6}}\right\vert &\leq &\frac{13}{48}.
\end{eqnarray*}%
All of the above estimates, except that on ${a_{6},}$ are sharp.
\end{theorem}

\begin{proof}
Since $f\in \mathcal{S}_{G}^{\ast }$, there exists an analytic function $w$
with $w(0)=0$ and $|w(z)|<1$ in $\mathcal{U}$ such that%
\begin{equation}
\frac{zf^{\prime }(z)}{f(z)}=\Psi (w(z))=\frac{w(z)}{\ln \left(
1+w(z)\right) }\text{ \ \ }(z\in \mathcal{U}).  \label{2.1}
\end{equation}%
Define the functions $p$ by%
\begin{equation*}
p(z)=\frac{1+w(z)}{1-w(z)}=1+p_{1}z+p_{2}z^{2}+\cdots \text{ \ \ }(z\in 
\mathcal{U})
\end{equation*}%
or equivalently,
\begin{eqnarray}
w(z) &=&\frac{p(z)-1}{p(z)+1}=\frac{p_{1}}{2}z+\frac{1}{2}\left( p_{2}-\frac{%
p_{1}^{2}}{2}\right) z^{2}+\frac{1}{2}\left( p_{3}-p_{1}p_{2}+\frac{p_{1}^{3}%
}{4}\right) z^{3}  \label{2.2} \\
&&+\frac{1}{2}\left( p_{4}-p_{1}p_{3}+\frac{3p_{1}^{2}p_{2}}{4}-\frac{%
p_{2}^{2}}{2}-\frac{p_{1}^{4}}{8}\right) z^{4}  \notag \\
&&{+\frac{1}{2}\left( p_{5}-\frac{1}{2}p_{1}^{3}p_{2}+\frac{3}{4}%
p_{1}p_{2}^{2}+\frac{3}{4}p_{1}^{2}p_{3}-p_{2}p_{3}-p_{1}p_{4}+\frac{1}{16}%
p_{1}^{5}\right) z^{5}+}\cdots  \notag
\end{eqnarray}
in $\mathcal{U}.$ Then $p$ is analytic in $\mathcal{U}$ with $p(0)=1$ and
has positive real part in $\mathcal{U}$. By using (\ref{2.2}) together with $%
\frac{w(z)}{\ln \left( 1+w(z)\right) },$ it is evident that \small
\begin{eqnarray}
\Psi (w(z)) &=&{1+\frac{p_{1}z}{4}+\frac{1}{48}\left(
-7p_{1}^{2}+12p_{2}\right) z^{2}}{+\frac{1}{192}\left(
17p_{1}^{3}-56p_{1}p_{2}+48p_{3}\right) z^{3}}  \label{2.3} \\
&&+\frac{1}{11520}\left(
-649p_{1}^{4}+3060p_{1}^{2}p_{2}-3360p_{1}p_{3}-1680p_{2}^{2}+2880p_{4}%
\right) z^{4}  \notag \\
&&+\frac{1}{46080}\left( 
\begin{array}{l}
1739p_{1}^{5}-10384p_{1}^{3}p_{2}+12240p_{1}^{2}p_{3}+12240p_{1}p_{2}^{2} \\ 
-13440p_{2}p_{3}+13440p_{1}p_{4}+11520p_{5}%
\end{array}%
\right) z^{5}+\cdots .  \notag
\end{eqnarray} \normalsize
Since 
\begin{eqnarray}
\frac{zf^{\prime }(z)}{f(z)} &{=}&{1+a_{2}z+\left( -a_{2}^{2}+2a_{3}\right)
z^{2}+\left( a_{2}^{3}-3a_{2}a_{3}+3a_{4}\right) z^{3}}  \label{2.4} \\
&&{+\left( -a_{2}^{4}+4a_{2}^{2}a_{3}-2a_{3}^{2}-4a_{2}a_{4}+4a_{5}\right)
z^{4}}  \notag \\
&&{+\left(
a_{2}^{5}-5a_{2}^{3}a_{3}+5a_{2}a_{3}^{2}+5a_{2}^{2}a_{4}-5a_{3}a_{4}-5a_{2}a_{5}+5a_{6}\right) z^{5}+\cdots ,%
}  \notag
\end{eqnarray}%
it follows by (\ref{2.1}), (\ref{2.3})\textit{\ }and (\ref{2.4}) that 
\begin{eqnarray}
a_{2} &=&{\frac{p_{1}}{4}},  \label{aa2} \\
a_{3} &=&{\frac{1}{24}\left( -p_{1}^{2}+3p_{2}\right) },  \label{aa3} \\
{a_{4}} &{=}&{\frac{1}{288}\left( 4p_{1}^{3}-19p_{1}p_{2}+24p_{3}\right) ,}
\label{aa4} \\
{a_{5}} &{=}&-\frac{1}{16}\left( \frac{71}{720}p_{1}^{4}+\frac{11}{24}%
p_{2}^{2}+\frac{5}{6}p_{1}p_{3}-\frac{85}{144}p_{1}^{2}p_{2}-p_{4}\right) ,
\label{aa5} \\
{a_{6}} &{=}&{\frac{\left( 
\begin{array}{l}
2267p_{1}^{5}-15677p_{1}^{3}p_{2}+21720p_{1}^{2}p_{3}+23370p_{2}^{2}p_{1} \\ 
-29520p_{1}p_{4}-33120p_{2}p_{3}+34560p_{5}%
\end{array}%
\right) }{691200}.}  \label{aa6}
\end{eqnarray}
Thus, we get 
\begin{eqnarray*}
\left\vert a_{2}\right\vert &\leq &\frac{1}{2}~~\left( \text{from (\ref{eq24}%
)}\right) , \\
\left\vert a_{3}\right\vert &=&\frac{1}{8}\left\vert p_{2}-\frac{1}{3}%
p_{1}^{2}\right\vert \leq \frac{1}{4}~~\left( \text{from (\ref{eq26})}%
\right) , \\
\left\vert a_{4}\right\vert &{=}&\frac{1}{12}\left\vert p_{3}-\frac{19}{24}%
p_{1}p_{2}+\frac{1}{6}p_{1}^{3}\right\vert \leq \frac{1}{6}~~\left( \text{%
from (\ref{eq25})}\right) , \\
\left\vert {a_{5}}\right\vert &{=}&\frac{1}{16}\left\vert \frac{71}{720}%
p_{1}^{4}+\frac{11}{24}p_{2}^{2}+\frac{5}{6}p_{1}p_{3}-\frac{85}{144}%
p_{1}^{2}p_{2}-p_{4}\right\vert \leq \frac{1}{8}~~\left( \text{from (\ref%
{eq28})}\right) .
\end{eqnarray*}
We now find the estimates on $\left\vert {a_{6}}\right\vert .$ Therefore,
from the Lemma \ref{lmma3}, $\left\vert p_{n}\right\vert \leq 2$ and (\ref%
{eq26}), we find that \small
\begin{eqnarray*}
\left\vert {a_{6}}\right\vert &{=}&\frac{\left| 
\begin{array}{l}
2267 p_{1}^{5}-15677 p_{1}^{3}p_{2}+21720p_{1}^{2}p_{3} +23370p_{2}^{2}p_{1}
\\ 
-29520p_{1}p_{4}-33120p_{2}p_{3}+34560p_{5}%
\end{array}
\right| }{691200} \\
&\leq & \frac{41\left|p_1\right|}{960}\left|\frac{2267}{29520}p_1^4+\frac{779%
}{984}p_2^2+\frac{181}{246}p_1p_3-\frac{15677}{29520}p_1^2p_2-p_4 \right|+%
\frac{1}{20}\left|p_5-\frac{23}{24}p_2p_3\right| \\
&\leq &\frac{41}{240}+\frac{1}{10}=\frac{13}{48}.
\end{eqnarray*} \normalsize
The first five results are sharp for the function $f:\mathcal{U\rightarrow 
\mathbb{C}}$ given by 
\begin{eqnarray}
f_1(z) &=& z\exp \left( \int_{0}^{z}\frac{\Psi (t)-1}{t}dt\right)=z+\frac{1}{%
2}z^{2}+\cdots,  \label{sharp1} \\
f_2(z) &=& z\exp \left( \int_{0}^{z}\frac{\Psi (t^2)-1}{t}dt\right)=z+\frac{1%
}{4}z^{3}+\cdots,  \label{sharp2} \\
f_3(z) &=& z\exp \left( \int_{0}^{z}\frac{\Psi (t^3)-1}{t}dt\right)=z+\frac{1%
}{6}z^{4}+\cdots,  \label{sharp3} \\
f_4(z) &=& z\exp \left( \int_{0}^{z}\frac{\Psi (t^4)-1}{t}dt\right)=z+\frac{1%
}{8}z^{5}+\cdots,  \label{sharp4} \\
f_5(z) &=& z\exp \left( \int_{0}^{z}\frac{\Psi (t^5)-1}{t}dt\right)=z+\frac{1%
}{10}z^{6}+\cdots.  \label{sharp5}
\end{eqnarray}
This completes the proof of Theorem \ref{t1}.
\end{proof}

\begin{conjecture}
$\allowbreak $\label{c1} Let $f(z)=z+a_{2}z^{2}+a_{3}z^{3}+\cdots \in 
\mathcal{S}_{G}^{\ast }.$ As we see from the extremal functions $%
f_{i}(z),~~i=1,2,3,4$, the first five coeffitients are obtained as our sharp
bound estimates for $\left\vert {a_{n}}\right\vert ,$ $n=2,3,4,5.$ By
comparing with the extremal function $f_{5}(z),$ the sixth coefficient is
expected as the sharp bound estimate for $\left\vert {a_{6}}\right\vert .$
So it is an open question whether the bound for $\left\vert {a_{6}}%
\right\vert $ is sharp and if one can show the following estimate $%
\left\vert {a_{6}}\right\vert \leq \frac{1}{10}$. Also, it is an open
question the sharp inequality $\left\vert a_{n}\right\vert \leq \frac{1}{%
2(n-1)}$ for all $n\in 
\mathbb{N}
\backslash \{1\}.$
\end{conjecture}

In 2021, Deniz \cite{Deniz} defined a subclass of starlike functions which
is related with generalized telephone numbers. He obtained some sharp bounds
of coefficients for this class. These results are as follows: $\left\vert
a_{2}\right\vert \leq {1},$ $\left\vert a_{3}\right\vert \leq {1,}$ $%
\left\vert {a_{4}}\right\vert {\leq }{\frac{8}{9},}$ $\left\vert {a_{5}}%
\right\vert {\leq }\frac{107}{144}$ and $\left\vert {a_{6}}\right\vert \leq 
\frac{2381}{3600}.$ When we compare these results with Theorem \ref{t1}, seen that our results are better.

\begin{theorem}
$\allowbreak $\label{t2} Let $f(z)=z+a_{2}z^{2}+a_{3}z^{3}+\cdots \in 
\mathcal{S}_{G}^{\ast }.$ Then, the following sharp estimates holds:%
\begin{equation*}
\left\vert {a}_{3}-\mu {a_{2}^{2}}\right\vert \leq \frac{1}{4}\max
\{1,\left\vert \mu -\frac{1}{3}\right\vert \},\text{ \ \ }\left( \mu \in 
\mathbb{C}
\right)
\end{equation*}%
\begin{equation*}
\left\vert {a_{2}a}_{3}-a_{4}\right\vert \leq \frac{1}{6}
\end{equation*}%
and%
\begin{equation*}
\left\vert {a_{2}a}_{4}-a_{3}^{2}\right\vert \leq \frac{1}{16}.
\end{equation*}
\end{theorem}

\begin{proof}
From (\ref{aa2}) and (\ref{aa3})\textit{, }we have%
\begin{equation*}
{a}_{3}-\mu {a_{2}^{2}}={\frac{1}{24}\left( -p_{1}^{2}+3p_{2}\right) }-\mu {%
\frac{p_{1}^{2}}{16}}=\frac{1}{8}\left( p_{2}-\left( \frac{3\mu +2}{6}%
\right) p_{1}^{2}\right) .
\end{equation*}%
If we take $\nu =\frac{3\mu +2}{6}$ in known result (see \cite{Keo}): $%
\left\vert {p}_{2}-\nu {p_{1}^{2}}\right\vert \leq 2\max \{1,\left\vert 2\nu
-1\right\vert \}$ for $\mu ,\nu \in 
\mathbb{C}
,$ we obtain that%
\begin{equation*}
\left\vert {a}_{3}-\mu {a_{2}^{2}}\right\vert \leq \frac{1}{4}\max
\{1,\left\vert \mu -\frac{1}{3}\right\vert \}.
\end{equation*}%
Similarly from (\ref{aa2}), (\ref{aa3}) and (\ref{aa4}), we have
\begin{equation*}
{a_{2}a}_{3}-a_{4}{=}-\frac{1}{288}\left(
24p_{3}-28p_{1}p_{2}+7p_{1}^{3}\right)
\end{equation*}%
and so from (\ref{eq25})
\begin{equation*}
\left\vert {a_{2}a}_{3}-a_{4}\right\vert {\leq }\frac{1}{288}\left\vert
24p_{3}-28p_{1}p_{2}+7p_{1}^{3}\right\vert =\frac{1}{12}\left\vert p_{3}-%
\frac{7}{6}p_{1}p_{2}+\frac{7}{24}p_{1}^{3}\right\vert \leq \frac{1}{6}.
\end{equation*}

Now, we investigate last estimate in Theorem \ref{t2}. From (\ref{aa2}), (%
\ref{aa3}) and (\ref{aa4}) again, we see that%
\begin{equation*}
{a_{2}a}_{4}-a_{3}^{2}=\frac{%
2p_{1}^{4}-7p_{1}^{2}p_{2}-18p_{2}^{2}+24p_{1}p_{3}}{1152}:=T,
\end{equation*}%
which, upon applying Lemma \ref{lmma1} and assuming that $s=p_{1}\in \lbrack
0,2]$, we can write 
\begin{equation*}
T=\frac{4-s^{2}}{2304}\left( -s^{2}x-3\left( s^{2}+12\right) x^{2}+24s\left(
1-\left\vert x\right\vert ^{2}\right) y.\right)
\end{equation*}%
If $s=0,$ then $T=-\frac{1}{16}x^{2}.$ Thus, since $\left\vert x\right\vert
\leq 1$, we have 
\begin{equation}
\left\vert T\right\vert \leq \frac{1}{16}.  \label{L29}
\end{equation}%
If $s=2,$ then 
\begin{equation}
\left\vert T\right\vert =0.  \label{L30}
\end{equation}

We now let $s\in (0,2).$ Then, we can write%
\begin{eqnarray*}
\left\vert T\right\vert &=&\left|\frac{4-s^2}{2304}\left(-s^2x-3\left(s^2+12%
\right)x^2+24s\left(1-\left|x\right|^2\right)y. \right)\right| \\
&\leq& \frac{s\left(4-s^2\right)}{96}\left[\left|-\frac{s}{24}x-\frac{s^2+12%
}{8s}x^2\right|+1-\left\vert x\right\vert ^{2} \right] \\
&=&\frac{s\left(4-s^2\right)}{96}\left[ \left\vert \widetilde{a}+\widetilde{b%
}x+\widetilde{c}x^{2}\right\vert +1-\left\vert x\right\vert ^{2}\right]
\end{eqnarray*}
where 
\begin{equation*}
\widetilde{a}=0,\;\widetilde{b}=-\frac{s}{24},\;\widetilde{c}=-\frac{
s^{2}+12}{8s}.
\end{equation*}
It follows that $\widetilde{a}\widetilde{c}=0.$ Also, we easily see that 
\begin{equation*}
\left\vert \widetilde{b}\right\vert -2(1-\left\vert \widetilde{c}\right\vert
)=\frac{7s^{2}-48s+72}{24s}>0.
\end{equation*}
Therefore, we have%
\begin{eqnarray*}
\left\vert T\right\vert &\leq & \frac{s\left(4-s^2\right)}{96}\left(
\left\vert \widetilde{a}\right\vert +\left\vert \widetilde{b}\right\vert
+\left\vert \widetilde{c}\right\vert \right)=\frac{s\left(4-s^2\right)}{96}%
\left(\frac{s^2+9}{6s} \right) \\
&=&\frac{-s^4-5s^2+36}{576}.
\end{eqnarray*}
Let $t=s^{2}$ $\left( t\in (0,4)\right) .$ Then, we investigate maximum of
the function $H_{0}$ defined by%
\begin{equation*}
H_{0}(t)=\frac{-t^2-5t+36}{576}.
\end{equation*}%
In that case, we have 
\begin{equation}
\left\vert T\right\vert \leq \left\vert H_{0}(t)\right\vert \leq \frac{%
-t^2-5t+36}{576}\leq \frac{1}{16}.  \label{L31}
\end{equation}
\end{proof}

\begin{remark}
\label{R2} By using Theorem \ref{t1} and Theorem \ref{t2}, we have 
\begin{equation*}
\left\vert H_{3}(1)\right\vert \leq \left\vert a_{3}\right\vert \left\vert
a_{2}a_{4}-a_{3}^{2}\right\vert +\left\vert a_{4}\right\vert \left\vert
a_{4}-a_{2}a_{3}\right\vert +\left\vert a_{5}\right\vert \left\vert
a_{3}-a_{2}^{2}\right\vert \leq \frac{43}{576}.
\end{equation*}
\end{remark}

If we compare the bound $\frac{43}{576}$ with the results of Zaprawa \cite%
{Zap} and Deniz \cite{Deniz}, see that this bound is better.

\subsection{Logarithmic Coefficients}

The following formula defines the logarithmic coefficients $\beta_n$ of $%
f(z)=z+\sum_{n=2}^{\infty}a_nz^n$ that belongs to $\mathcal{S}$ 
\begin{equation}  \label{eq313}
G_f\left(z\right):=\log \left(\frac{f\left(z\right)}{z}\right)=2\sum_{n=1}^{%
\infty}\beta_nz^n~~~ \text{for}~z\in \mathcal{U}.
\end{equation}
In many estimations, these coefficients provide a significant contribution
to the concept of univalent functions. In 1985, De Branges \cite{Branges}
proved that 
\begin{equation}  \label{eq314}
\sum_{k=1}^{n}k\left(n-k+1\right)\left|\beta_n\right|^2\leq \sum_{k=1}^{n}%
\frac{n-k+1}{k}~~~\forall~ n\geq 1
\end{equation}
and equality will be achieved if $f$ has the form $z/\left(1-e^{i\theta}z%
\right)^2$ for some $\theta \in \mathbb{R}.$ In its most comprehensive
version, this inequality offers the famous Bieber-bach-Robertson-Milin
conjectures regarding Taylor coefficients of $f\in \mathcal{S}.$ We refer to 
\cite{Avkhadiev,FitzGerald1985,FitzGerald1987} for further details on the
proof of De Branges finding. By considering the
logarithmic coefficients, Kayumov \cite{Kayumov} was able to prove Brennan's conjecture for conformal mappings in 2005. For your
reference, we mention a few works that have made major contributions to the
research of the logarithmic coefficients. Andreev and Duren \cite{Andreev},
Alimohammadi et al. \cite{Alimohammadi}, Deng \cite{Deng}, Roth \cite{Roth},
Ye \cite{Ye}, Obradovi\' c et al. \cite{Obradovic}, and finally the work of
Girela \cite{Girela} are the major contributions to the study of logarithmic
coefficients for different subclasses of holomorphic univalent functions.

As stated in the definition, it is simple to determine that for $f\in 
\mathcal{S},$ the logarithmic coefficients are computed by 
\begin{equation}  \label{eq315}
\beta_1=\frac{1}{2}a_2,
\end{equation}
\begin{equation}  \label{eq316}
\beta_2=\frac{1}{2}\left(a_3-\frac{1}{2}a_2^2\right),
\end{equation}
\begin{equation}  \label{eq317}
\beta_3=\frac{1}{2}\left(a_4-a_2a_3+\frac{1}{3}a_2^3\right),
\end{equation}
\begin{equation}  \label{eq318}
\beta_4=\frac{1}{2}\left(a_5-a_2a_4+a_2^2a_3-\frac{1}{2}a_3^2-\frac{1}{4}%
a_2^4\right).
\end{equation}

\begin{theorem}
\label{logT1} If $f\in \mathcal{S}^*_G$ and has the series representation $%
f(z)=z+\sum_{n=2}^{\infty}a_nz^n$ then 
\begin{equation}  \label{eq319}
\left|\beta_n\right| \leq \frac{1}{4n}~~~\left(n=1,2,3,4\right).
\end{equation}
These bounds are sharp and can be obtained from the extremal functions $%
f_i(z)~~i=1,2,3,4$ given by (\ref{sharp1})-(\ref{sharp4}).
\end{theorem}

\begin{proof}
Let $f\in \mathcal{S}^*_G.$ Then, putting (\ref{eq315}), (\ref{eq316}), (\ref%
{eq317}) and (\ref{eq318}) in (\ref{aa2}), (\ref{aa3}), (\ref{aa4}) and (\ref%
{aa5}), we get 
\begin{eqnarray}
\beta_1 &=&{\frac{p_{1}}{8}},  \label{ba1} \\
\beta_2 &=&{\frac{1}{192}\left( -7p_{1}^{2}+12p_{2}\right) },  \label{ba2} \\
\beta_3 &{=}&{\frac{1}{1152}\left( 17p_{1}^{3}-56p_{1}p_{2}+48p_{3}\right) ,}
\label{ba3} \\
\beta_4 &{=}&-\frac{1}{92160}\left( 649p_{1}^{4}+1680
p_{2}^{2}+3360p_{1}p_{3}-3060p_{1}^{2}p_{2}-2280p_{4}\right) .  \label{ba4}
\end{eqnarray}
For $\beta_1,$ using (\ref{eq24}) in(\ref{ba1}), we obtain 
\begin{equation*}
\left|\beta_1\right|\leq \frac{1}{4}.
\end{equation*}
For $\beta_2,$ putting (\ref{eq26}) in(\ref{ba2}), we get 
\begin{equation*}
\left|\beta_2\right|\leq \frac{1}{8}.
\end{equation*}
For $\beta_3,$ we can rewrite (\ref{ba3}) as 
\begin{equation*}
\left|\beta_3\right| =\frac{1}{24} \left|p_{3}-\frac{7}{6}p_{1}p_{2}+\frac{17%
}{48}p_{1}^{3}\right|,
\end{equation*}
using (\ref{eq25}) where $Q=\frac{7}{12}$ and $R=\frac{17}{48},$ we obtain 
\begin{equation*}
\left|\beta_3\right|\leq \frac{1}{12}.
\end{equation*}
For $\beta_4,$ we can rewrite (\ref{ba4}) as 
\begin{equation}  \label{eq324}
\beta_4 =-\frac{1}{32}\left( \frac{649}{2880}p_{1}^{4}+\frac{7}{12}
p_{2}^{2}+\frac{7}{6}p_{1}p_{3}-\frac{17}{16}p_{1}^{2}p_{2}-p_{4}\right).
\end{equation}
Comparing the right side of (\ref{eq324}) with inequality (\ref{eq28}),
where $\rho=\frac{649}{2880},$ $\varsigma=\frac{7}{12},$ $\tau=\frac{7}{12}$
and $\psi=\frac{17}{24}.$ It follows that 
\begin{equation*}
8\varsigma \left(1-\varsigma\right)\left[\left(\tau\psi-2\rho\right)^2+%
\left(\tau\left(\varsigma+\tau\right)-\psi \right)^2 \right]+\tau
\left(1-\tau\right)\left(\psi-2\varsigma\tau\right)^2=0.00442226
\end{equation*}
and 
\begin{equation*}
4\varsigma\tau^2\left(1-\tau\right)^2\left(1-\varsigma\right)=0.057435.
\end{equation*}
Using (\ref{eq27}) we deduce that 
\begin{equation*}
\left|\beta_4\right|\leq \frac{1}{16}.
\end{equation*}
\end{proof}

\subsection{Inverse Coefficients}

It is well-known that the function 
\begin{equation*}
f(z)=z+\sum_{n=2}^{\infty }a_{n}z^{n}\in \mathcal{S}
\end{equation*}%
has an inverse $f^{-1},$ which is analytic in $\left\vert w\right\vert <1/4,$
as we know from Koebe's $1/4$-theorem. If $f\in \mathcal{S},$ then 
\begin{equation}
f^{-1}(w)=w+A_{2}w^{2}+A_{3}w^{3}+\cdots ,~~~~\left\vert w\right\vert <1/4.
\label{eq330}
\end{equation}%
L\"{o}wner \cite{Lowner} proved that, if $f\in \mathcal{S}$ and its inverse
is given by (\ref{eq330}), then the sharp estimate 
\begin{equation}
\left\vert A_{n}\right\vert \leq \frac{\left( 2n\right) !}{n!\left(
n+1\right) !}  \label{eq331}
\end{equation}%
holds. It has been shown that the inverse of the Koebe function $%
k(z)=z/\left( 1-z\right) ^{2}$ provides the best bounds for all $\left\vert
A_{n}\right\vert ~~\left( n=2,3,\ldots \right) $ in (\ref{eq331}) over all
members of $\mathcal{S}.$ There has been a good deal of interest in
determining the behavior of the inverse coefficients of $f$ given in (\ref%
{eq330}) when the corresponding function $f$ is restricted to some proper
geometric subclasses of $\mathcal{S}.$ Alternate proofs of the inequality (%
\ref{eq331}) have been given by several authors but a simpler proof was
given by Yang \cite{Yang}.\newline
Since $f\left( f^{-1}(w)\right) =w,$ using (\ref{eq330}) it's very clear to
see 
\begin{equation}\label{eqson}
\begin{split}
A_{2}=& -a_{2}, \\
A_{3}=& 2a_{2}^{2}-a_{3}, \\
A_{4}=& -5a_{2}^{3}+5a_{2}a_{3}-a_{4}.
\end{split}%
\end{equation}

\begin{theorem}
\label{InvT1} If $f\in \mathcal{S}^*_G$ and has the series representation $%
f^{-1}(w)=w+A_2w^2+A_3w^3+\cdots$ then 
\begin{equation}  \label{eq325}
\begin{split}
\left|A_2\right|& \leq \frac{1}{2}, \\
\left|A_3\right|& \leq \frac{5}{12}, \\
\left|A_4\right|& \leq \frac{31}{72}.
\end{split}%
\end{equation}
These bounds are sharp, except $A_4.$
\end{theorem}

\begin{proof}
Let $f\in \mathcal{S}^*_G$ and $f^{-1}(w)=w+A_2w^2+A_3w^3+\cdots.$ If
equations (\ref{aa2}), (\ref{aa3}) and (\ref{aa4}) are written in (\ref{eqson}), we obtain 
\begin{align}
A_2&=-\frac{p_1}{4},  \label{eq334} \\
A_3&=\frac{p_1^2}{6}-\frac{p_2}{8},  \label{eq335} \\
A_4&=-\frac{1}{576}\left(48p_3-128p_1p_2+83p_1^3\right).  \label{eq336}
\end{align}
Because of $\left|p_1\right|\leq 2,$ the result for $\left|A_2\right|\leq 
\frac{1}{2}$ is trivial.\newline
We now estimate $\left|A_3\right|$ by using (\ref{eq26}) in equality (\ref%
{eq335}). Thus, we have 
\begin{equation*}
\left|A_3\right|\leq \frac{1}{8}\left|p_2-\frac{4}{3}p_1^2\right|\leq \frac{5%
}{12}.
\end{equation*}
From the relations Lemma \ref{lmma1} and (\ref{eq336}), and by some simple
calculations, we have 
\begin{equation*}
K=-\frac{4-\sigma^2}{24}\left(\frac{31\sigma^3}{24\left(4-\sigma^2\right)}-%
\frac{5\sigma}{3}x- \frac{\sigma}{2}x^2+\left(1-\left|x\right|^2\right)y
\right),
\end{equation*}
where $\sigma=p_1\in \left[0,2\right],$ $\left|x\right|\leq 1$ and $%
\left|y\right|\leq 1.$ We now investigate upper bound of the $\left|K\right|$
according to $\sigma.$ \newline
A. If $\sigma=0,$ then $T=-\frac{1}{6}\left(1-\left|x\right|^2\right)y$ and
so we have $\left|T\right|\leq \frac{1}{6}.$\newline
B. Let $\sigma=2.$ Then $T=-\frac{31}{72}$ and so we have $%
\left|T\right|\leq \frac{31}{72}.$\newline
C. We now assume that $\sigma \in \left(0,2\right). $ Then, we can write 
\begin{equation}  \label{eq337}
\begin{split}
\left|K\right|=&\frac{4-\sigma^2}{24}\left|\frac{31\sigma^3}{%
24\left(4-\sigma^2\right)}-\frac{5\sigma}{3}x- \frac{\sigma}{2}%
x^2+\left(1-\left|x\right|^2\right)y \right| \\
\leq & \frac{4-\sigma^2}{24}\left(\left|\widetilde{A}+\widetilde{B}x+%
\widetilde{C}x^2\right|+ 1-\left|x\right|^2 \right),
\end{split}%
\end{equation}
where 
\begin{equation*}
\widetilde{A}= \frac{31\sigma^3}{24\left(4-\sigma^2\right)},~~~ \widetilde{B}%
= -\frac{5\sigma}{3}\text{ \ and \ } \widetilde{C}= -\frac{\sigma}{2}.
\end{equation*}
For the rest of the proof, we use Lemma \ref{l3}. Then 
\begin{equation*}
\widetilde{A}\widetilde{C}=-\frac{31\sigma^4}{48\left(4-\sigma^2\right)}<0.
\end{equation*}
C1. Note that the inequality 
\begin{equation*}
-4\widetilde{A}\widetilde{C}\left(\widetilde{C}^{-2}-1\right)-\widetilde{B}%
^{2}=\frac{31\sigma^4}{12\left(4-\sigma^2\right)}\left( \frac{4}{\sigma^2}-1
\right)-\frac{25\sigma^2}{9}\leq0
\end{equation*}
which evidently holds for $s\in \left(-\infty,\infty\right).$ Moreover, the
inequality $\left|\widetilde{B}\right|<2\left(1-\left|\widetilde{C}%
\right|\right)$ is equivalent to $\sigma <\frac{3}{4},$ which is true for $%
\sigma\in \left(0,\frac{3}{4}\right).$ Then, by (\ref{eq337}) and Lemma \ref%
{l3}, 
\begin{equation}  \label{eq338}
\begin{split}
\left|K\right|\leq& \frac{4-\sigma^2}{24}\left(1-\left|\widetilde{A}\right|+%
\frac{\widetilde{B}^2}{4\left(1-\left|\widetilde{C}\right|\right)}\right) \\
=&\frac{7\sigma^3+128\sigma^2+288}{1728}\leq \frac{2581}{12288} \approx
0.210042
\end{split}%
\end{equation}
for $\sigma \in \left(0,\frac{3}{4}\right).$\newline
C2. Since 
\begin{equation*}
4\left(1+\left|\widetilde{C}\right|\right)^2=\sigma^2+4\sigma+4,\quad -4%
\widetilde{A}\widetilde{C}\left(\widetilde{C}^{-2}-1\right)=\frac{31\sigma^2%
}{12},
\end{equation*}
for $\sigma \in \left[\frac{3}{4},2\right),$ we see that the inequality 
\begin{equation*}
\frac{25\sigma^2}{9}=\widetilde{B}^{2}<\min\left\{4\left(1+\left|\widetilde{C%
}\right|\right)^2,~ -4\widetilde{A}\widetilde{C}\left(\widetilde{C}%
^{-2}-1\right)\right\}= \frac{31\sigma^2}{12}
\end{equation*}
is equivalent to $\frac{7\sigma^2}{36}<0,$ which is false for $\sigma \in %
\left[\frac{3}{4},2\right).$\newline
C3. Observe that the inequality 
\begin{equation*}
\left|\widetilde{C}\right|\left(\left|\widetilde{B}\right|+4\left|\widetilde{%
A}\right|\right)-\left|\widetilde{A}\widetilde{B}\right|=\frac{%
\sigma^2\left(29\sigma^2-240\right)}{72\left(\sigma^2-4\right)}\leq 0
\end{equation*}
is false for $\sigma \in \left[\frac{3}{4},2\right).$\newline
C4. Note that the inequality 
\begin{equation*}
\left|\widetilde{A}\widetilde{B}\right|-\left|\widetilde{C}%
\right|\left(\left|\widetilde{B}\right|-4\left|\widetilde{A}\right|\right)=%
\frac{\sigma^2\left(401\sigma^2-240\right)}{72\left(4-\sigma^2\right)}\leq 0
\end{equation*}
is true $\sigma \in \left[\frac{3}{4},4\sqrt{\frac{15}{401}}\right].$ Then,
by (\ref{eq337}) and Lemma \ref{l3}, 
\begin{equation}  \label{eq339}
\begin{split}
\left|K\right|\leq& \frac{4-\sigma^2}{24}\left(-\left|\widetilde{A}%
\right|+\left|\widetilde{B}\right|+\left|\widetilde{C}\right|\right) \\
=&\frac{1}{576}\left(208\sigma-83\sigma^3\right)\leq \frac{3968}{1203}\sqrt{%
\frac{5}{1203}} \approx 0.212646
\end{split}%
\end{equation}
for $\sigma \in \left[\frac{3}{4},4\sqrt{\frac{15}{401}}\right].$\newline
C5. It remains to consider the last case in Lemma \ref{l3}, for $\sigma \in
\left(4\sqrt{\frac{15}{401}},2\right).$ Then, by (\ref{eq337}) 
\begin{equation}  \label{eq340}
\begin{split}
\left|K\right|\leq& \frac{4-\sigma^2}{24}\left(\left|\widetilde{A}%
\right|+\left|\widetilde{C}\right|\right)\sqrt{1-\frac{\widetilde{B}^2}{4%
\widetilde{A}\widetilde{C}}} \\
=&\frac{\sqrt{400-7\sigma^2}\left(19\sigma^2+48\right)}{576\sqrt{93}}\leq 
\frac{31}{72} \approx 0.430556
\end{split}%
\end{equation}
is true for $\sigma \in \left(4\sqrt{\frac{15}{401}},2\right).$ In addition,
as can be seen from the figure below, the function $H_1(\sigma)=\frac{\sqrt{%
400-7\sigma^2}\left(19\sigma^2+48\right)}{576\sqrt{93}}$ is increasing in $%
\left(4\sqrt{\frac{15}{401}},2\right)$ and takes its maximum value at the
limit.

\begin{figure}[h]
\begin{center}
\begin{tabular}{l}
\includegraphics[scale=0.7]{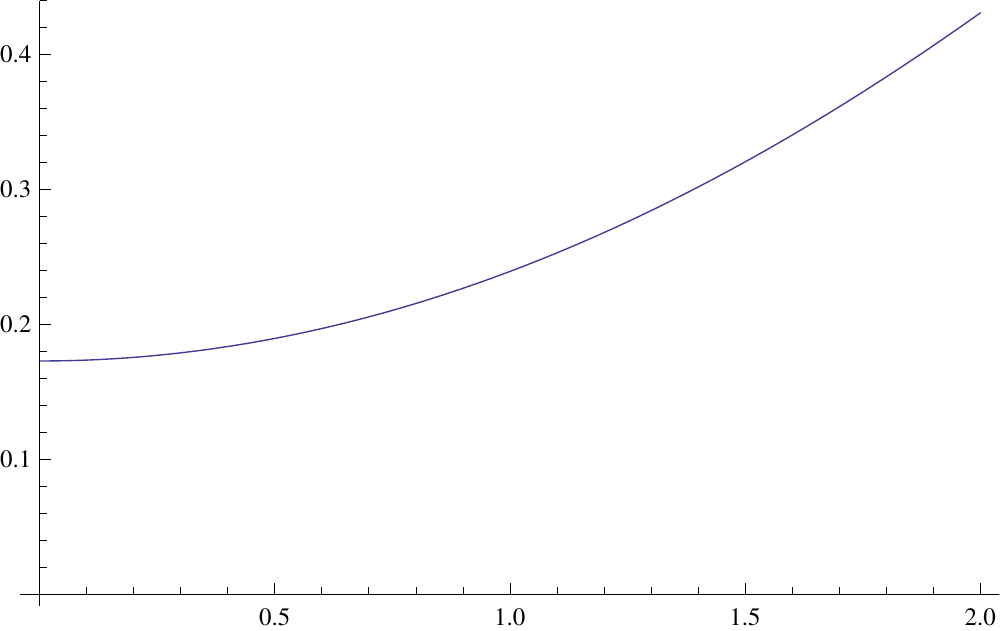}%
\end{tabular}%
\end{center}
\caption{The graph of the function $H_1(\protect\sigma)$ on $%
\left(0,2\right) $}
\label{fig:Figure1}
\end{figure}
From A-C, we have $\left|A_3\right|\leq \frac{31}{72}.$
\end{proof}

\end{document}